\documentclass[a4paper,11pt]{amsart}

\usepackage[utf8]{inputenc}
\usepackage{amsfonts,amssymb}
\usepackage{amsmath}

\usepackage{tikz}
\usetikzlibrary{decorations.pathreplacing,calc,intersections,arrows}
\theoremstyle{plain}
\newtheorem{theorem}{Theorem}[section]
\newtheorem{lemma}[theorem]{Lemma}
\newtheorem{prop}[theorem]{Proposition}
\newtheorem{cor}[theorem]{Corollary}
\newtheorem{conjecture}[theorem]{Conjecture}
\theoremstyle{definition}

\newtheorem{definition}[theorem]{Definition}

\newcommand\SL{\mathrm{SL}}
\newcommand\PSL{\mathrm{PSL}}

\newcommand\SP{\mathrm{SP}}
\newcommand\SO{\mathrm{SO}}
\newcommand\SU{\mathrm{SU}}
\newcommand\R{\mathbf{R}}
\newcommand\CC{\mathbf{C}}
\newcommand\cC{\mathcal{C}}
\newcommand\cE{\mathcal{E}}
\newcommand\cG{\mathcal{G}}
\newcommand\F{\mathbf{F}}
\newcommand\free{\mathbb{F}}

\newcommand\Q{\mathbf{Q}}
\newcommand\N{\mathbf{N}}

\newcommand\Z{\mathbf{Z}}
\newcommand\cH{\mathcal{H}}
\newcommand\LL{\mathcal{L}}

\numberwithin{equation}{section}

\begin{document}


\title{Analysis on simple Lie groups and lattices}

\author[M.~De~La~Salle]{Mikael De La Salle}
\address{CNRS, Université de Lyon, France}
\thanks{MdlS was funded by the ANR grants AGIRA ANR-16-CE40-0022 and Noncommutative analysis on groups and quantum groups ANR-19-CE40-0002-01}


\begin{abstract}
We present a simple tool to perform analysis with groups such as $\SL_n(\R)$ and $\SL_n(\Z)$, that has been introduced by Vincent Lafforgue in his study of non-unitary representations, in connection with the Baum-Connes conjecture  and strong property (T). It has been later applied in various contexts: operator algebras, Fourier analysis, geometry of Banach spaces or dynamics. The idea is to first restrict to compact subgroups and then exploit how they sit inside the whole group.
\end{abstract}

\maketitle


This text is devoted to the presentation of a single idea that has
been useful to answer several analysis questions on higher rank simple
Lie groups and lattices. This idea originates from Vincent Lafforgue's
work \cite{MR2423763}, and can be summarized as \emph{rank $0$
  reduction}.

One of the basic tools to study Lie algebras is that of
$\mathfrak{sl}_2$ triples. In the context a semisimple Lie groups,
it is often used in the following form: to understand a possibly
complicated Lie groups, one restricts to its subgroups locally
isomorphic to $\SL_2(\R)$ (and there are plenty by the Jacobson-Morozov theorem), the
simplest non compact semisimple Lie group. This idea, that we could
call \emph{rank $1$ reduction} because $\SL_2(\R)$ has rank $1$, can
be powerful. For example, it allows to obtain precise information on
the unitary representations of higher rank simple Lie groups
\cite{MR1151617,MR1905394}. But, as we shall see later in this text,
there are situations where rank $1$ reduction is not
efficient. Rank $2$ reduction has also become a standard tool in the study of higher rank simple Lie groups, as every real simple Lie group of rank $\geq 2$ contains a subgroup locally isomorphic to one of the rank $2$ groups $\SL_3(\R)$ of $\SP_2(\R)$. See for example \cite[I.1.6]{MR2415834}.

Rank $0$ reduction, the subject of this survey, consists in studying a Lie group through its compact subgroups. The idea is to first restrict to the compact subgroups of the Lie group $G$ and do analysis there. It is perhaps surprising that there are nontrivial general things to say (see Proposition~\ref{prop:Uinvariant_K_coeff}). And then, in a second step, by analyzing the relative positions of the various cosets of compact groups in $G$, it is possible to promote the local phenomena that have been discovered in the first step to global phenomena in $G$.

In the whole text except for the brief and last
section~\ref{sec:otherGroups}, we only consider for simplicity the Lie
group $\SL_3(\R)$ and its subgroup $\SL_3(\Z)$. In the first section,
we illustrate rank $0$ reduction in the simplest meaningful setting of
unitary representations of $\SL_3(\R)$~: we shall see that this idea
provides a new proof of Kazhdan's celebrated theorem that $\SL_3(\R)$
has property (T). In the next sections we show several other
applications of this idea for $\SL_3(\R)$ or $\SL_3(\Z)$, each time
with a brief history of the problems. In Section~\ref{sec:fourier},
devoted to Fourier analysis for non-commutative groups, we explain how
this same idea can show that Fourier synthesis (the reconstruction of
a ``function'' from its Fourier series) is in a way impossible for
$\SL_3(\R)$ and $\SL_3(\Z)$. We then interpret these Fourier analysis
statements in terms of approximation properties of the von Neumann
algebra of $\SL_3(\Z)$ and its noncommutative $L_p$ spaces. In
Section~\ref{sec:strongT} is devoted to strong property (T)~: we study
non unitary representations of $\SL_3(\R)$ and $\SL_3(\Z)$ on Hilbert
spaces and see that the same idea allows to prove some form of
property (T) for them. Applications of strong property (T) are also
described, in particular we explain how strong property (T) appears as
a key tool in the resolution by Brown, Fisher and Hurtado of
Zimmer's conjecture for group actions of high rank lattices on
manifolds of small dimension. Section~\ref{sec:BanachSpaces} is
devoted to group actions on Banach spaces~: we investigate how much of
strong property (T) can be proved for representations on more general
Banach spaces. Finally in Section~\ref{sec:otherGroups} we survey how these ideas have been used for other semisimple Lie groups or algebraic groups over other local fields.

\section{A proof of property (T) for $\SL_3(\R)$}\label{sec:SL3}
Throughout this text, by \emph{representation} of a locally compact group $G$ on a Banach space $X$, we will always mean a group homomorphism $\pi$ from $G$ to the group of invertible continuous linear maps on $X$ that is continuous for the strong operator topology: for every $\xi \in X$, $g \mapsto \pi(g) \xi$ is continuous. A unitary representation is when $\pi$ takes values in the unitary group of a Hilbert space.

A topological group $G$ has Kazhdan's property (T) whenever the trivial representation is isolated in its unitary dual for the Fell topology. This means that every unitary representation $\pi$ of $G$ on a Hilbert space $\cH$ which almost has invariant vectors (\emph{i.e.} there is a net $\xi_i$ of unit vectors in $\cH$ such that $\lim_i\|\pi(g) \xi_i - \xi_i\|=0$ uniformly on compact subsets of $G$), has a nonzero invariant vector. Because of its numerous applications, property (T) has become a central concept in many areas of mathematics such as geometric and analytic group theory, operator algebras, ergodic theory, etc. See \cite{MR2415834}.

The purpose of this introductory section is to give a detailed proof of Kazhdan's celebrated theorem \cite{MR0209390} that the group $\SL_3(\R)$ has property (T). We do not give one of the classical proofs (which rely in a way or another on the pair $\SL_2 \subset \SL_3$) but a proof due to Lafforgue \cite{MR2423763} (which relies on the pair $\SO(3) \subset \SL_3(\R)$). We denote $G=\SL_3(\R)$, $K = \SO(3) \subset G$ the maximal compact subgroup. Then the polar decomposition asserts that every element of $G$ can be written as a product $g=k a k'$ for $k,k' \in K$ and a diagonal matrix $a$ with positive entries in non-increasing order. In other word, it identifies the double classes $K \backslash G/K$ with the \emph{Weyl chamber} $\Lambda =\{ (r,s,t) \in \R^3, r\geq s \geq t, r+s+t=0\}$ via the identification of $(r,s,t)$ with the class $KD(r,s,t) K$ of 
\[ D(r,s,t) = \begin{pmatrix} e^r&0&0\\ 0&e^s&0\\0&0&e^t\end{pmatrix}.\]
We also introduce the subgroup $U \simeq \SO(2) \subset K$ of block-diagonal matrices
\[ U = \left\{ \begin{pmatrix} 1 & 0 &0\\0&*&*\\0&*&*\end{pmatrix} \right\} \cap K.\]
The double classes $U \backslash K/U$ are then parametrized by $[-1,1]$, the parametrization being given by $UkU \mapsto k_{1,1}$. For every $\delta \in [-1,1]$, let $k_\delta$ denote a representative of the corresponding double class, for example
\[ k_\delta = \begin{pmatrix} \delta & -\sqrt{1-\delta^2} & 0 \\ \sqrt{1-\delta^2} & \delta& 0\\0&0&1 \end{pmatrix}.\]
\paragraph{Step 1.} The first step is the observation that $U$-biinvariant matrix coefficients of unitary representations of $K$ are H\"older $\frac 1 2$-continuous on $(-1,1)$. In particular
\begin{prop}\label{prop:Uinvariant_K_coeff} For every unitary representation $\pi$ of $K$ on a Hilbert space $\mathcal H$ and every $\pi(U)$-invariant unit vectors $\xi,\eta \in \mathcal H$ we have
\[ | \langle \pi(k_\delta) \xi,\eta\rangle - \langle \pi(k_0) \xi,\eta\rangle| \leq 2 \sqrt{|\delta|}.\]
\end{prop}
\begin{proof}
By the Peter-Weyl theorem it is enough to prove the inequality for the
irreducible representations of $\SO(3)$. For the $n$-th irreducible
representation of $\SO(3)$ (the degree $n$ spherical harmonics) the
quantity $\langle \pi(k_\delta) \xi,\eta\rangle$ is equal to $\pm
P_n(\delta)$, the value at $\delta$ of the $n$-th Legendre polynomial
normalized by $P_n(1)=1$, see for example \cite{MR2426516}. So we have
to prove that $\sup_n |P_n(\delta) - P_n(0)|\leq 2
\sqrt{|\delta|}$. By bounding
\[ |P_n(\delta) - P_n(0)| \leq \min( |P_n(0)| +|P_n(\delta)|, |\delta| \max_{t \in [0,1]} |P_n'(t\delta)|)\]
and using the Bernstein inequality $|P_n(x)| \leq \min(1,\sqrt{\frac{2}{\pi n}} (1-x^2)^{-\frac 1 4})$ \cite[Theorem 7.3.3]{MR0372517} and the formula
\[ (1-x^2) P'_n(x) = -nx P_n(x)+nP_{n-1}(x)\]
expressing $P'_n$ in terms of $P_n$ and $P_{n-1}$ one deduces the
proposition.\end{proof}

\paragraph{Step 2.} The next step is to deduce regularity properties of $K$-biinvariant matrix coefficients of unitary representations of $G$. The proof is short, but there are important things happening.
\begin{prop}\label{prop:Kinvariant_coeff_Cauchy} Let $\pi$ be a unitary representation of $G$ on a Hilbert space $\cH$, and $\xi,\eta$ be $\pi(K)$-invariant unit vectors. Then for every $g_1,g_2 \in G$
\[ |\langle \pi(g_1) \xi,\eta\rangle - \langle \pi(g_2) \xi,\eta\rangle| \leq 100 \min(\|g_1\|,\|g_1^{-1}\|,\|g_2\|,\|g_2^{-1}\|)^{-\frac 1 2}.\]
\end{prop}
\begin{proof} We may regard the matrix coefficient $g \mapsto \langle \pi(g) \xi,\eta\rangle$ as the map $c\colon \Lambda \to \CC$ given by $c(r,s,t) = \langle \pi(D(r,s,t)) \xi,\eta\rangle$. For every $(r,s,t) \in \Lambda$, the matrix $D(-t,\frac{t}{2},\frac{t}{2})$ commutes with $U$, so the unit vectors $\pi(D(-t,\frac{t}{2},\frac{t}{2}))\xi$ and $\pi(D(t,-\frac{t}{2},-\frac{t}{2})) \eta$ are $U$-invariant. We can therefore apply Proposition \ref{prop:Uinvariant_K_coeff}. With $\delta = \frac{\sinh(r+\frac t 2)}{\sinh(-\frac{3t}{2})}$, we obtain
  \begin{equation}\label{eq:horizontal_estimates} |c(r,s,t) - c(-\frac{t}{2},-\frac{t}{2},t)| \leq 2 \sqrt{\delta} \leq 2 e^{\frac{r}{2} + t} = 2e^{-\frac r 2-s}.\end{equation}
Applying this to the representation $g \mapsto \pi( (g^T)^{-1})$, we also obtain
  \begin{equation}\label{eq:vertical_estimates} |c(r,s,t) - c(r,-\frac{r}{2},-\frac{r}{2})| \leq 2 e^{\frac{t}{2}+s}.\end{equation}
  These two inequalities are best understood on a picture (see Figure~\ref{figure:zigzag})~: \eqref{eq:horizontal_estimates} expresses that the amplitude of $c$ is very small (exponentially small in the distance to the origin) on lines of slope $-\frac 1 2$ in the region $s\geq -1$, whereas \eqref{eq:vertical_estimates} expresses that the amplitude of $c$ is very small on vertical lines in the region $s\leq 0$. We can join any two points of the Weyl chamber by a zig-zag path as in Figure \ref{figure:zigzag}, and combining these estimates we deduce
  \[ |c(r,s,t)-c(r',s',t')| \leq 100 \max( e^{-\frac{\min(r,-t)}{2}}, e^{-\frac{\min(r',-t')}{2}}),\]
  which is exactly the proposition.
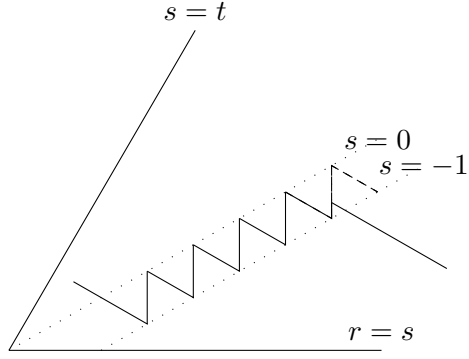
\begin{figure}[t]
  \center
\begin{tikzpicture}[scale=.7]
\draw (0,0)--(0:7) node[above] {$r=s$};
\draw (0,0)--(60:7)  node[above] {$s=t$};
\draw[shift=(-30:1), loosely dotted] (30:1)--(30:8) node {$s=-1$};
\draw[loosely dotted] (0,0)--(30:8) node {$s=0$};
 \foreach \x in {0,1,2,3}
  {
\draw[shift=(30:2+\x)] (0,0)--(-30:1)--(30:1);
  }
\draw[shift=(30:2)] (0,0)--(150:.6);
\draw[shift=(30:2+5), dashed] (-30:1)--(0,0)--(-90:1);
\draw[shift=(30:2+4), dashed] (0,0)--(-30:1)--(30:1);
\draw[shift=(30:2+5), dashed] (0,0)--(-30:1);
\begin{scope}[shift={(30:2+4)}] 
\begin{scope}[shift={(-30:1)}] 
\draw (150:1)--(0,0)--(90:0.3);
\draw[shift=(90:0.3)] (0,0)--(-30:2.5); 
\end{scope}
\end{scope}
\end{tikzpicture}
\caption{The zig-zag path in the Weyl chamber $\Lambda$.} \label{figure:zigzag}
\end{figure}
\end{proof}
If $\pi$ is a representation of $G$ on a Hilbert space $\cH$ and $g \in G$, let us denote $\pi(KgK)$ the bounded operator on $\cH$ mapping $\xi$ to $\iint_{K \times K} \pi(kgk')\xi dk dk'$, where the
integrals are with respect to the Haar probability measure on $K$. We also say that a vector $\xi \in \cH$ is \emph{harmonic} if $\pi(KgK)\xi = \xi$ for every $g \in G$.  The terminology is justified by \cite{MR47056}.
\begin{cor}\label{cor:convergence_to_harminonic_projection} If $\pi$ is a unitary representation of $G$, then
  \begin{equation}\label{eq:cor:convergence_to_projection_on_harmonics} \| \pi(KgK) - P \|_{B(\cH)} \leq 100 \min(\|g\|,\|g^{-1}\|)^{-\frac 1 2},
  \end{equation}
  where $P$ is a projection on the space of harmonic vectors.
\end{cor}
\begin{proof} Taking the supremum over all $K$-invariant unit vectors in Proposition~\ref{prop:Kinvariant_coeff_Cauchy}, we obtain that $(\pi(KgK))_{g \in G}$ is Cauchy in $B(\cH)$, and that its limit $P$ satisfies \eqref{eq:cor:convergence_to_projection_on_harmonics}. $P$ is clearly the identity on the space of harmonic vectors, and the following computation shows that the image of $P$ is made of harmonic vectors,
  \[ \pi(KgK) P \xi dk = \lim_{g'\to \infty} \pi(KgK) \pi(Kg'K)\xi = \lim_{g'\to \infty} \int_K \pi(K gkg' K) \xi = P\xi.\]
  So $P$ is indeed a projection (and even \emph{the} orthogonal projection) on the space of harmonic vectors.
\end{proof}

\paragraph{Step 3.} The last step, which can be summarized as \emph{harmonic implies invariant}, is the conclusion of the proof of property (T). Let $\pi$ be a unitary representation of $G$ with almost invariant vectors. Evaluating $\pi(KgK)$ at almost invariant vectors, we see that $\pi(KgK)$ has norm $1$ for every $g \in G$. The limit $P$ in Corollary~\ref{cor:convergence_to_harminonic_projection} therefore also has norm $1$, which means that there is a nonzero harmonic vector $\xi$. For every $g \in G$, the equality $\xi = \iint \pi( k gk')\xi dk dk'$ expresses $\xi$ as an average of vectors of the same norm $\pi(kg'k)\xi$. By strict convexity of Hilbert spaces, we have that $\pi(kgk')\xi = \xi $ for every $k,k'$, in particular $\xi$ is $\pi(g)$-invariant. So $\xi$ is a nonzero invariant vector. This proves that $G=\SL_3(\R)$ has property (T).

For further reference, we can rephrase
Proposition~\ref{prop:Uinvariant_K_coeff} in terms of the
operators $T_\delta\colon L_2(S^2) \to L_2(S^2)$ defined by $T_\delta
f(x) $ is the average of $f$ on the circle $\{y \in S^2 | \langle
x,y\rangle = \delta\}$. Identifying $S^2$ with
$\SO(3)/\SO(2)$, we see that
Proposition~\ref{prop:Kinvariant_coeff_Cauchy} is equivalent to
\begin{equation} \label{eq:norm_of_Tdelta}\|T_\delta - T_0\| \leq 2
\sqrt{|\delta|}.
\end{equation}


\subsection{Comments on the proofs}

The distinct steps in the proof of propery (T) for $\SL_3(\R)$ have
different nature. The first step is analytic and deals with harmonic
analysis on a compact group, and more precisely on the pair $(U
\subset K)$ of compact groups. The second step is
geometric/combinatorial. What happens is that one studies the various
$K$-equivariant embeddings of the sphere $S^2$ (identified with $K/U =
\SO(3)/\SO(2)$) into the symmetric space $G/K = \SL_3(\R) /
\SO(3)$. The relative position between a pair of such embeddings gives
rise to an embedding of $U\backslash K/U = [-1,1]$ inside the Weyl
chamber $K \backslash G /K$. There are three types of such embeddings:
the segments in Figure \ref{figure:zigzag} ; others, that are of no
use for us, are segments parallel to the line $s=0$. Combining these
embeddings allows us to explore the whole Weyl chamber of $\SL(3,\R)$,
and to prove that $K$-biinvariant matrix coefficients of unitary
representations of $G$ are H\"older $\frac 1 2$-continuous away from
the boundary of $\Lambda$. The crucial fact that leads to
\eqref{eq:horizontal_estimates} and \eqref{eq:vertical_estimates}
expresses that the interesting embeddings are exponentially distorted
in the distance to the origin in the Weyl chamber, and hence that this
exploration allows us to escape to infinity \emph{in finite time}. The
last step is rather obvious, but will become much more involved later.

It is informative to make similar computations for rank $1$ simple Lie
groups $G$ which contain a subgroup isomorphic to $\SO(3)$ (for
example for $\SO(3,1)$). In that case, one gets also lots of
embeddings of $[-1,1]$ inside the Weyl chamber $[0,\infty)$ of $G$,
  and also enough to explore the whole Weyl chamber and prove H\"older
  $\frac 1 2$-regularity in the interior of the Weyl chamber. The
  difference (and the reason why this does not create a contradiction
  by proving that $\SO(3,1)$ has property (T)!) is that these
  embeddings are almost isometric, and so it takes an infinite time to
  explore the whole Weyl chamber. In a sense, only the segments going
  straight but slowly to infinity exist in rank one. Those that are
  used in the zig-zag argument, that take less direct routes but are
  faster, only appear in higher rank.

The fact that all the analysis is done at the level of the compact groups $U,K$ is very important, because harmonic analysis for compact groups is much better understood than for arbitrary groups (see for example the very easy results in Lemma \ref{lem:A_M0A_compact}, \ref{lem:FellBanach}, \ref{lem:TwoStepRepsOfCompactGroups}). This is what permits to use a similar approach for other objects than coefficients of unitary representations, and to prove rigidity results in various other linear settings. This is the content of the remaining of this survey.

\subsection{Induction and property (T) for $\SL_3(\Z)$}\label{subsec:induction}
For later reference, we recall the classical argument why property (T) for $\SL_3(\R)$ implies property (T) for $\SL_3(\Z)$. We uses Minkowski's theorem that $\SL_3(\Z)$ is a lattice in $\SL_3(\R)$. A lattice in a locally compact group $G$ is a discrete subgroup such that the quotient $G/\Gamma$ carries a finite $G$-invariant Borel probability measure. Equivalently, there is a Borel probability measure $\mu$ on $G$ whose image in $G/\Gamma$ is $G$-invariant. The proof that property (T) passes to lattices uses induction of representations. If $\pi$ is a unitary representation of $\Gamma$ on a Hilbert space $\cH$, the space of the induced representation is the space $L_2(G,\mu;\cH)^\Gamma$ of measurable functions $G \to \cH$ that satisfy $f(g\gamma) = \pi(\gamma^{-1}) f(g)$ for every $g \in G$ and $\gamma$ in $\Gamma$, and such that $\int_G \|f(g)\|^2 d\mu(g) < \infty$. It is equipped with a unitary representation of $G$ by left-translation $g \cdot f = f(g^{-1}\cdot)$. If $\pi$ almost has invariant vectors, then so does the induced representation. By property (T) for $G$, it has a non zero invariant vector. This vector is a constant function with values in $\cH^\pi$. So $\pi$ has invariant nonzero vectors.

\section{Fourier series, approximation properties, operator algebras}\label{sec:fourier}
\subsection{Fourier series, absence of Fourier synthesis}
If $1<p<\infty$, it is very well-known (this follows immediately from Marcel Riesz's theorem that the Hilbert transform $f \mapsto \sum_{n \geq 0} \hat f(n) e^{2i\pi \cdot}$ is bounded on $L_p$) that the Fourier series of every function $f \in L_p(\R/\Z)$ converges in $L_p$:
\[ \lim_N \|f - \sum_{n=-N}^N \hat f(n) e^{2i\pi \cdot}\|_p = 0.\]
This is not true for $p=1$ or $p=\infty$, but there are more clever \emph{summation methods}, for example Fej\'er's method: if we denote $W_N(n) = \max(1-\frac{|n|}{N},0)$, then
\[ \lim_N \|f - \sum_n W_N(n) \hat f(n) e^{2i\pi \cdot}\|_p = 0,\]
and this time the convergence holds for $L_1$, and even $L_\infty$ if $f$ is continuous. All this remains true on the torus $(\R/\Z)^d$ of arbitrary dimension, and more generally on the Pontryagin dual $\hat \Gamma$ of every dicrete abelian group.

When $\Gamma$ is a non-abelian discrete  group, $\hat \Gamma$ does not make sense as a group, but the space $C(\hat \Gamma)$, $L_\infty(\hat \Gamma)$ and $L_p(\hat \Gamma)$ have a very natural meaning: they are respectively the reduced $C^*$-algebra $C^*_\lambda(\Gamma)$, the von Neumann algebra $\LL\Gamma$ and the non-commutative $L_p$ space \cite{MR1999201} of the von Neumann algebra of $\LL \Gamma$. Recall that if $\lambda$ is the left regular representation (by left-translation) of $\Gamma$ on $\ell_2(\Gamma)$, $C_\lambda^*(\lambda)$ is the norm closure in $B(\ell_2(\Gamma))$ of the linear span of $\lambda(\Gamma)$, $\LL\Gamma$ is its closure for the weak-operator topology, and $L_p(\LL \Gamma)$ its completion for the norm $x \mapsto \langle |x|^p \delta_e,\delta_e\rangle^{\frac 1 p}$. The elements of each of these spaces admit a Fourier series $f= \sum_\gamma \hat f(\gamma) \lambda(\gamma)$. This is a formal series, whose convergence is not clear in general (except in $L_2$). One can wonder in that case whether, as in the case when $\Gamma$ is abelian, there exist Fourier summation methods in this context, \emph{i.e.} a sequence of functions $W_N \colon \Gamma \to \CC$ with finite support such that 
\begin{equation}\label{eq:Fourier} f = \lim_N \sum_\gamma W_N(\gamma) \hat f(\gamma) \lambda(\gamma),\end{equation}
for every $f$ in $C^*_\lambda(\Gamma)$ or $L_p(\LL\Gamma)$ (convergence in norm). It is well known that this is the case when $\Gamma$ is amenable. According to a celebrated (and surprising at that time) result by Haagerup \cite{MR520930}, this is also true when $\Gamma$ is a non-abelian free group. This result has inspired a massive research program on approximation properties of group operator algebras. Indeed, by the Banach-Steinhaus theorem, if \eqref{eq:Fourier} holds, then the maps $f \mapsto \sum_\Gamma W_N(\gamma)  \hat f(\gamma) \lambda(\gamma)$, called Fourier multipliers, are uniformly bounded, have finite rank and converge pointwise to the identity. In \cite{MR784292} de Cannière and Haagerup realized that in the case of free groups, the convergence even holds for every $f$ in the $C^*$-algebra $B(\ell^2) \otimes C^*_\lambda(\Gamma)$, that is when the Fourier coefficients $\hat f(\gamma)$ are bounded operators on $\ell^2$. When this holds, $\Gamma$ is said to be weakly amenable. Moreover, a function $W\colon \Gamma \to \CC$ is said to be a completely bounded Fourier multiplier if the map $f \mapsto \sum_\gamma W(\gamma) \hat f(\gamma) \otimes \lambda(\gamma)$ is bounded on $B(\ell^2) \otimes C^*_\lambda(\Gamma)$. So weak amenability comes with a constant, which is the smallest common upper bound on these completely bounded norms, among all sequences of finitely supported multipliers achieving \eqref{eq:Fourier}. Mutatis mutandis, completely bounded Fourier multipliers and weak amenability also make sense for locally compact groups, and importantly restriction to closed subgroups and induction from lattices work well for completely bounded Fourier multipliers. In particular, the weak amenability constant coincide for a group and a lattice \cite{MR3476201}. The major achievement in this direction was obtained in a series a work by Haagerup with de Canni\`ere and Cowling \cite{MR784292,MR996553,MR3476201}, that lead to the exact computation of the weak amenability constant of all simple Lie groups. Excluding exceptional groups, it is equal to $1$ for $\SO(n,1)$, $\SU(n,1)$, $2n-1$ for $\SP(n,1)$ and is infinite for higher rank groups.

In \cite{MR1220905}, Haagerup and Kraus discovered a strange phenomenon~: it might happen that there is no sequence $W_N$ satisfying \eqref{eq:Fourier} for every $f \in B(\ell^2) \otimes C^*_\lambda(\Gamma)$, but there exists such a generalized sequence (or net).\footnote{Remembering that the Banach-Steinhaus theorem is false for nets might help imagine how such a statement could be true. An example is given by $\SL_2(\Z) \ltimes \Z^2$: it has the AP as the semidirect product of two weakly amenable groups and AP is stable by group extensions \cite{MR1220905}. That it is not weakly amenable was proved in \cite{MR3476201} and also reproved in \cite{MR2914879}.} A group for which such a net exists is said to have \emph{the approximation property} of Haagerup and Kraus, or simply AP. 

So groups without the AP are groups in which no $L_\infty$-summation method exists whatsoever for operator coefficients. It has been difficult to produce such groups. It was known \cite{MR1220905} that non-exact groups \cite{zbMATH05057460} would be such examples, but they are difficult to construct. Haagerup et Kraus had conjectured that $\SL_3(\Z)$ was another example. This conjecture turned out to be delicate, because classical approaches to rigidity in higher-rank lattices, which rely on the subgroup $\SL_2(\Z) \ltimes \Z^2$ or its relative $\SL_2(\R) \ltimes \R^2$, are inefficient by the strange phenomenon described above. It is the ideas described in Section~\ref{sec:SL3} that allowed to settle it.

\begin{theorem}\label{thm:nonAPSL3}\cite{LaffdlS} $\SL_3(\Z)$ does not have the approximation property of Haagerup and Kraus.
\end{theorem}
The original proof of this theorem was not direct and went through $L_p$ Fourier theory (Theorem~\ref{thm:nonAPLpSL3}), but thanks to several simplifications \cite{MR3047470,vergarapAP,habilitation}, a very easy proof is now known. We have already justified that, by induction, we can as well prove the Theorem for $\SL_3(\R)$. The starting point is the following straightfoward lemma (which is known to hold more generally if \cite{MR0040592} and only if \cite{MR806070} the locally compact group $K$ is amenable), which provides a characterization of completely bounded Fourier multipliers of compact groups.
\begin{lemma}\label{lem:A_M0A_compact} Let $K$ be a compact group. A function $\varphi \colon K \to \CC$ defines a completely bounded Fourier multiplier of $C^*_\lambda(K)$ if and only if $\varphi$ is a matrix coefficient of a unitary representation of $K$.
\end{lemma}
Proposition~\ref{prop:Uinvariant_K_coeff} therefore says that $\SO(2)$-biinvariant completely bounded Fourier multipliers of $C^*_\lambda(\SO(3))$ are H\"older $\frac 1 2$ continuous in the interior of $\SO(2)\backslash\SO(3)/\SO(2)$. The same proof as Proposition~\ref{prop:Kinvariant_coeff_Cauchy} then produces a map from $\SO(3)\backslash\SL_3(\R)/\SO(3)$ into the dual of the space of completely bounded Fourier multipliers of $\SL_3(\R)$, which satisfies the Cauchy criterion. Its limit vanishes on compactly supported functions and takes the value $1$ on the identity multiplier. This is exactly a Hahn-Banach type separation, which says that $\SL_3(\R)$ does not have the AP. 

The same argument can also be used to say something about $L^p$ Fourier summability for some finite $p$ and to obtain the following, which is an equivalent form of the main result in \cite{LaffdlS}.
\begin{theorem}\label{thm:nonAPLpSL3}\cite{LaffdlS} For every $4<p<\infty$ or $1 \leq p < \frac 4 3$, there is $f \in L_p(\LL \SL_3(\Z) \otimes B(\ell_2))$ such that, for every finitely supported $W:\SL_3(\Z) \to \CC$,
\[ \| f - \sum_\gamma W(\gamma) \hat f(\gamma) \lambda(\gamma)\|_{p} \geq 1.\]
\end{theorem}
Again, the proof has the same structure as in Section~\ref{sec:SL3}. To initiate the first step, we need to investigate in more details the spectral decomposition of the operators $T_\delta$ in \eqref{eq:norm_of_Tdelta}. For a Hilbert space $\cH$, the Schatten $p$-class $S_p(\cH)$ is the space of operators $T$ on $\cH$ such that $\|T\|_{S_p}:=( \operatorname{Tr}( |T|^p))^{\frac 1 p}<\infty$. It can be shown \cite{LaffdlS} that the operators $T_\delta$, $\delta \in (-1,1)$ belong to $S_p(L_2(S^2))$ if $p>4$, and there is a constant $C$ such that for every $\delta,\delta' \in [-1/2,1/2]$, 
\begin{equation}\label{eq:Spestimate_Tdeltaupper} \|T_\delta - T_{0}\|_{S_p} \leq \frac{C}{(p-4)^{1/p}} |\delta|^{\frac{1}{2} - \frac{2}{p}}.\end{equation}
With this inequality, we can run the argument of Section~\ref{sec:SL3} and obtain a form of Theorem~\ref{thm:nonAPLpSL3} for $\SL_3(\R)$. However, the induction procedure for completely bounded $L_p$ Fourier multipliers, which works well when $p=\infty$ \cite{MR1220905,MR3476201}, is problematic for $1<p\neq 2<\infty$. It is known to work well mainly for amenable groups \cite{CaspdlS15,MR3482270,MR2866074}. The solution is to work with Herz-Schur multipliers on $S_p(L_2(G))$, that is operators of the form $A = (A_{g,h}) \in S_p(L_2(G)) \mapsto (W(gh^{-1}) A_{g,h})$ for functions $W\colon G\to \CC$, for which induction works well. Fortunately, for compact groups (and even amenable groups \cite{CaspdlS15,MR2866074}) Schur $S_p$ and Fourier $L_p$ multipliers coincide.

The preceding sketch is the only proof I know of Theorem~\ref{thm:nonAPLpSL3}. It can be shown that \eqref{eq:Spestimate_Tdeltaupper} is optimal, and that $T_\delta - T_0$ does not belong to $S_4$ for any $\delta\neq 0$. So this idea cannot work for $2 <p\leq 4$, and it is an intriguing question whether this condition is really needed for Theorem~\ref{thm:nonAPLpSL3}. As we will explain in Section~\ref{sec:otherGroups}, a positive answer to this question would allow to distinguish the von Neumann algebra of $\SL_3(\Z)$ and $\PSL_n(\Z)$ for $n >4$, and would confirm a conjecture of Connes. The challenge is to construct nontrivial $L_p$ Fourier multipliers for $\SL_3(\Z)$. A first step is to do so for $\SL_3(\R)$. Together with Parcet and Ricard \cite{ParcetRicarddlS}, we made some progress on that recently by proving a satisfactory \emph{local} form of the H\"ormander-Mikhlin multiplier theorem for $\SL_n(\R)$.

\subsection{Approximation properties for Banach spaces and operator algebras}\label{subsection:approximation_properties}

In his thesis \cite{MR0075539}, Grothendieck initiated the study of tensor product of topological vector spaces, and realized the tight connection with the Banach's approximation property AP: a Banach space $E$ has the approximation property (AP) if the identity operator belongs to the closure of the space of finite rank operators, for the topology of uniform convergence on compact sets. He was even led to conjecture that all Banach space have the AP (this would make the theory of tensor product simpler!). Later in his \emph{R\'esum\'e} \cite{MR94682} he changed his mind and actually expected that Banach spaces exist, which fail the AP. R\'ecoltes et Semailles \cite{RecoltesSemailles} contains a moving part, where Grothendieck explains how much he suffered from the year he spent on working on this problem without progress. It was only much later than the first example of Banach space was \emph{constructed} by Enflo \cite{MR402468}, followed by many other examples. See the survey \cite{MR1863695} for a list. Let me emphasize: the examples of Enflo, as most other examples, are obtained from delicate combinatorial constructions, and in particular they are not Banach space, the existence of which was known to Grothendieck. There is, to my knowledge, only one example of Banach space that is both natural (not obtained from an \emph{ad hoc} construction) and known to fail the AP, namely the space $B(\ell_2)$ of all bounded operators on $\ell_2$ \cite{MR631090}. This space is not reflexive and we are lacking separable and natural examples of space with the AP. For some time in the 1970's, a candidate of such a satisfactory example was $C^*_\lambda(\free_2)$, the reduced $C^*$-algebra on the non-abelian free group with two generators, and the hope was that the lack of AP could be explained by the non amenability of $\free_2$. Haagerup broke this hope in \cite{MR520930} by proving, better, that $C^*_\lambda(\free_2)$ has the metric AP (the identity belongs to the closure of the norm $1$ finite rank operators). More precisely, as explained in the previous section, $\free_2$ is weakly amenable with constant $1$. However, two serious candidates of natural separable Banach spaces without the AP remain in the same vein: $C^*(\free_2)$, the full $C^*$-algebra of $\free_2$ and $C^*_\lambda(\SL_3(\Z))$. It is hoped that the first lacks the AP for the simple reason that $\free_2$ is non-amenable, and the latter for reasons related to the ideas in Section~\ref{sec:SL3}.

This last conjecture has not been settled, but its weakening in the sense of operator spaces is known. Indeed, there are natural variants of Grothendieck's approximation property in the category of operator spaces rather than Banach spaces (replacing bounded operators by completely bounded operators \cite{MR2006539}), that we denote by OAP and CBAP. The CBAP, meaning that the identity on $E$ belongs to the closure of the finite rank operators with completely bounded norm bounded above by some constant $C$, comes with a constant (the best such $C$). Given a discrete group $\Gamma$, the fundamental observation of Haagerup \cite{MR3476201} (and Kraus for OAP \cite{MR1220905}) is that the CBAP/OAP for $C^*_\lambda(\Gamma)$ (and respectively its variants for dual spaces for the von Neumann algebra $\LL(\Gamma)$) can be achieved with finitely supported Fourier multipliers. As a consequence, the weak amenability constant of a discrete group $\Gamma$ is an invariant of its von Neumann algebra, and so is the AP of Haagerup and Kraus. For example, the computation of the weak amenability constant of simple Lie groups \cite{MR784292,MR996553,MR3476201} discussed above allows to distinguish the von Neumann algebra of a lattice in $\SP(n,1)$ (for $n\neq 3, n\neq 11$) from the von Neumann algebra of a lattice in a simple Lie group that is not locally $\SP(n,1)$. Also, Theorem~\ref{thm:nonAPSL3} can be rephrased as \emph{$C^*_\lambda(\SL_3(\Z))$ does not have the approximation property in the category of operator spaces}. This was the first example of an exact $C^*$-algebra without the OAP. Similarly, Theorem~\ref{thm:nonAPLpSL3} can be rephrased as \emph{$L_p(\LL \SL_3(\Z))$ does not have the OAP for $p>4$}. For $4<p<80$, this was the first example of a non-commutative $L_p$ space without the CBAP.

\section{Non unitary representations : Lafforgue's strong property (T)}\label{sec:strongT}
This section is devoted to non unitary representations on Hilbert spaces.
\subsection{Strong property (T) for $\SL_3(\R)$}
It is useful at this point to have a look back at the proof of property (T) for $\SL_3(\R)$ in Section~\ref{sec:SL3} and see precisely where the assumption that the representations are unitary was used. Step 1 is about the compact group $\SO(3)$, and therefore only uses that the restriction to $\SO(3)$ is unitary. This is not a very strong assumption, as compact groups are unitarizable:
\begin{lemma}\label{lem:FellBanach} Every representation of a compact group on a Hilbert space is similar to a unitary representation.
\end{lemma}

In particular, every representation of $\SL_3(\R)$ is similar to a representation whose restriction to $\SO(3)$ is unitary.

Step 2 did not use that $\pi$ is unitary in a strong way, and remains true (with a different constant and the exponent $-\frac 1 2$ replaced by $2\alpha-\frac{1}{2}$) if
\[ \alpha =  \alpha(\pi) := \limsup_{|t|\to \infty} \frac{1}{|t|}\log \| \pi( D(-t,\frac t 2,\frac t 2))\|< \frac 1 4.\]
So if $\pi$ is a representation of $\SL_3(\R)$ on a Hilbert space such that $\alpha(\pi)<\frac 1 4$, we obtain that $\pi(KgK)$ converges in norm to a projection on the space of harmonic vectors. Step 3 apparently relies more fundamentally on the fact that $\pi$ is unitary. However it is still true when $\alpha(\pi)<\frac 1 4$ that harmonic vectors are invariant. This requires new ideas, which make step 3 the most involved part, see \cite{MR2423763} or the presentation in \cite{dlSIsrael15}. To summarize, the condition $\alpha(\pi)<\frac 1 4$ is enough to guarantee that $\pi(KgK)$ converges in norm to a projection on the invariant vectors. In the terminology of Vincent Lafforgue \cite{MR2423763}, $\SL_3(\R)$ has strong property (T), which is a form of property (T) for representations on Hilbert spaces with small exponential growth rate. Let us spell out the definition.

If $G$ is a locally compact function, a length function $\ell : G\to \R_+$ is a function that is bounded on compact subsets, that is symmetric $\ell(g) = \ell(g^{-1})$ and subadditive $\ell(gh)\leq \ell(g) + \ell(h)$. Let us denote $\cC_\ell(G)$ the completion of $C_c(G)$ for the norm given by $\|f\|_{\cC_\ell(G)} = \sup \|\pi(f)\|$ where the supremum is over all representations of $G$ on a Hilbert space such that $\|\pi(g)\| \leq e^{\ell(g)}$ for every $g \in G$. It is a Banach algebra for the convolution. An element $p \in \cC_\ell(G)$ is called a \emph{Kazhdan projection} if $\pi(p)$ is a projection on the invariant vectors $\cH^\pi$ of $\pi$ for every such representation. Kazhdan projections have been investigated in \cite{MR4000570,dlSlocalKazhdan}. If a Kazhdan projection exists, $\pi(p)$ is the unique $G$-equivariant projection on $\cH^\pi$, so $p$ is unique \cite{dlSlocalKazhdan}.
\begin{definition}\label{def:strongT} \cite{MR2423763} $G$ has strong property (T) if for every length function $\ell$, there is $s>0$ such that for every $C>0$, $\cC_{s \ell +C}(G)$ has a Kazhdan projection.
\end{definition}
A posteriori, a group with strong property (T) is necessarily compactly generated (this is even a consequence of property (T)), and it is enough to check the definition for $\ell$ the word-length with respect to a compact generating set.

We have just explained the proof of the following theorem.
\begin{theorem}\cite{MR2423763}\label{thm:strongTSL3R} $\SL_3(\R)$ has strong property (T).
\end{theorem}
Moreover the Kazhdan projection belongs to the closure of $C_c(G)_+:=\{f \in C_c(G) \mid f(g) \geq 0 \forall g \in G\}$. Some authors \cite{brownFisherHurtado,brownFisherHurtado2,brownFisherHurtado3,brownDamjanovicZhang} add this precision to the definition because, as we will see below, this is crucial for some applications.

Vincent Lafforgue's original motivation for this definition was his work on the Baum-Connes conjecture. Indeed, strong property (T) (and variants of it) are the natural obstructions for applying some of his ideas (and in particular the ideas in \cite{MR2874956}) to groups such as $\SL_3(\Z)$. We refer to \cite{MR2732057} for more on the link with the Baum-Connes conjecture.

\subsection{Strong property (T) for $\SL_3(\Z)$}
Theorem~\ref{thm:strongTSL3R} also holds for $\SL_3(\Z)$, but the proof turned out to be delicate. In particular, we do not see how to prove in general that strong property (T) passes to lattices.
\begin{theorem}\cite{dlSActa}\label{thm:strongTSL3Z} $\SL_3(\Z)$ has strong property (T).
\end{theorem}
The way Theorem~\ref{thm:strongTSL3Z} is proved is by introducing and working with representation-like objects, where one is only allowed to compose once and that I call \emph{two-step representations}.
\begin{definition} A two-step representation of a topological group $G$ is a tuple $(X_0,X_1,X_2,\pi_0,\pi_1)$ where $X_0,X_1,X_2$ are Banach spaces and $\pi_i\colon G \to B(X_i,X_{i+1})$ are strongly continuous\footnote{\emph{i.e.} for every $x \in X_i$, the map $g \in G \mapsto \pi(g)x \in X_{i+1}$ is continuous.} maps such that
  \[ \pi_1(gg') \pi_0(g'') = \pi_1(g) \pi_0(g' g'')\textrm{ for every }g,g',g'' \in G.\]
In this case we will denote by $\pi \colon G \to B(X_0,X_2)$ the continuous map satisfying $\pi(gg')=\pi_1(g) \pi_0(g')$ for every $g,g' \in G$.
\end{definition}
The reason for this introduction is that two-step representations appear naturally when we induce non unitary representations. Indeed, let $\Gamma \subset G$ be a lattice. For every probability measure $\mu$ on $G$ as in Subsection~\ref{subsec:induction}, we can consider the space $L_2(G,\mu;\cH)^\Gamma$.  
When $\Gamma$ is a cocompact lattice (that is $G/\Gamma$ is compact), Lafforgue \cite{MR2423763} observed that it is possible to choose $\mu$ in such a way that the induced representation remains by bounded operators on $L_2(G,\mu;\cH)^\Gamma$, with small exponential growth if the original representation of $\Gamma$ had small exponential growth. Therefore, strong property (T) passes to cocompact lattices \cite{MR2423763}. When $\Gamma$ is not cocompact and $\pi$ is not uniformly bounded, there does not seem to be any choice of $\mu$ for which the representation is by well-defined (bounded) operators. However, in the particular case of $\SL_3(\Z) \subset \SL_3(\R)$, if $\mu$ is well-chosen and if the original representation $\pi$ of has small enough exponential growth, it possible to show that the representation $g \cdot f=f(g^{-1} \cdot)$ is bounded from $L_p(G,\mu;\cH)^\Gamma \to L_q(G,\mu;\cH)^\Gamma$ whenever $\frac 1 q - \frac 1 p \geq \frac 1 2$. This uses some strong exponential integrability properties of $\SL_3(\Z)$ in $\SL_3(\R)$, which rely on the celebrated Lubotzky-Mozes-Raghunathan theorem \cite{MR1244421}. In particular, we obtain a two-step representation with $X_0 =  L_{\infty}(G,\mu;\cH)^\Gamma$, $X_1 =  L_{2}(G,\mu;\cH)^\Gamma$ and $X_2 = L_{1}(G,\mu;\cH)^\Gamma$. So Theorem~\ref{thm:strongTSL3Z} is a consequence of a form of strong property (T) for two-step representations of $\SL_3(\R)$ where $X_1$ is a Hilbert space. This is done following the strategy in Section~\ref{sec:SL3}. The starting point is again a straightforward statement, which asserts that two-step representations of compact groups are governed by usual representations:
\begin{lemma}\label{lem:TwoStepRepsOfCompactGroups} If $(X_0,X_1,X_2,\pi_0,\pi_1)$ is a two-step representation of a compact group $K$ where $X_1$ is a Hilbert space, then there is a constant $C$ such that, for every $f \in C_c(K)$
  \[ \|\pi(f)\|_{B(X_0,X_2)} \leq C \|\lambda(f)\|_{C^*_\lambda(K)}.\]
\end{lemma}

\subsection{Applications of Strong property (T)}
Let us end this section with two applications of strong property (T). We will see in the next section other applications of variants of strong property (T) for representations on Banach spaces. All applications have in common that strong (T) is used as a way to systematically find and locate fixed point. The first is a result about vanishing of first cohomology spaces for representations with small exponential growth. If $\pi$ is a representation of $G$ on a space $\cH$, we denote by $H^1(G,\pi)$ the quotient of the space of cocycles
\[Z^1(G,\pi):= \{ b \in C(G,\cH) \mid \forall g_1,g_2 \in G,b(g_1g_2)= b(g_1) + \pi(g_1) b(g_2)\}\]
by the subspace of coboundaries
\[ B^1(G,\pi) = \{ g \mapsto \pi(g) \xi - \xi \mid \xi \in \cH\}.\]
Hence $H^1(G,\pi)$ parametrizes the continuous affine actions of $G$ on $\cH$ with linear part $\pi$, up to a change of origin. Delorme and Guichardet have proved that a second countable locally compact group $G$ has property (T) if and only if $H^1(G,\pi)=0$ for every unitary representation $\pi$. The following result has the same flavour, but the proof is different. The idea is that any $b \in Z^1(G,\pi)$ gives rise to a representation $\begin{pmatrix} \pi(g) & b(g) \\ 0 &1 \end{pmatrix}$ with the same exponential growth rate on $\cH\oplus \CC$.
\begin{prop}\label{prop:H1}\cite{MR2574023} If $G$ has strong property and $\ell$ is a length function on $G$, there is $s>0$ such that $H^1(G,\pi) = 0$ for every representation with exponential growth rate $\leq s$.
\end{prop}
This proposition can be used to show that strong property (T) is incompatible with hyperbolic geometry, and in particular that infinite Gromov-hyperbolic groups do not have strong property (T). Indeed, in \cite{MR2423763}, for every group $G$ acting with infinite orbits on a Gromov-hyperbolic graph with bounded degree, a representation with quadratic growth rate on a Hilbert space is constructed with $H^1(G,\pi) \neq 0$. By Proposition~\ref{prop:H1}, such group cannot have strong property (T).

Another notable application of strong property (T) was found in the resolution of most cases of Zimmer's conjecture by Brown, Fisher and Hurtado (see the texts by Aaron Brown and David Fisher in these proceedings). The following is a particular case of their result.
\begin{theorem}\cite{brownFisherHurtado}\label{thm:bfh} Let $G$ be a locally compact group with length function $\ell$ and $\alpha:G \to \operatorname{Diff}(M)$ an action by $C^\infty$ diffeomorphisms on a compact Riemannian manifold with subexponential growth of derivatives:
  \[ \forall \varepsilon>0, \sup_{g \in G} e^{-\varepsilon \ell(g)} \sup_{x \in M} \|D_x\alpha(g)\|<\infty.\]
  If $G$ has strong property (T) with Kazhdan projections in the closure of $C_c(G)_+$, then for every $k$, $\alpha$ preserves $C^k$ Riemannian metrics on $M$. In particular for the original metric, $\alpha$ has bounded derivatives.
\end{theorem}
The idea is to use strong property (T) for the representation on the Hilbert space of signed metrics on $M$ with Sobolev norms $W^{n,2}$, which take into account the $L^2$-norms of all derivatives of order $\leq n$. Here $n$ is an arbitrary positive integer. Strong property (T) allows to construct such invariant signed metrics. The fact that the Kazhdan projection belongs to the closure of nonnegative functions is used to ensure that these signed metrics can be taken positive. The Sobolev embedding theorems say that, for $n$ large, these metrics become smoother.

\section{Banach space representations}\label{sec:BanachSpaces}
\subsection{Banach spaces versions of property (T)}
The last two decades have seen important developpments in the study of group actions on Banach spaces, initiated by a number of more or less simultaneous investigations \cite{MR1914617,MR2221161,MR2316269,MR2423763,MR2671183,MR2421319}.

The work \cite{MR2316269} (and also \cite{MR2671183,MR2423763}) have proposed to study different possible generalizations of property (T) with Hilbert spaces replaced by Banach spaces. If one adopts the definition in terms of almost invariant vectors, one gets property (T$_X$). Let $\cE$ be a class of Banach spaces.
\begin{definition}(Bader, Furman, Gelander, Monod \cite{MR2316269}) A locally compact group $G$ has property (T$_\cE$) if for every isometric representation $\pi \colon G \to O(X)$ on a space $X$ in $\cE$, 
the quotient representation $G \to O(X/X^{\pi(G)})$ does not almost have invariant vectors. Here, $X^{\pi(G)}$ denotes the closed subspace of $X$ consisting of vectors that are fixed by $\pi$.
\end{definition}
When the quotient representation $G \to O(X/X^{\pi(G)})$ does not almost have invariant vectors, we say that $\pi$ has spectral gap. 

Compact groups have T$_X$ with respect to all Banach spaces, but for other locally compact groups we have to impose conditions on a Banach space to hope to have T$_X$. For example, the action by left-translation of $C_0(G)$ has spectral gap if and only $G$ is compact.

Adapting the equivalent characterization of property (T) in cohomological terms, we obtain:
\begin{definition}(\cite{MR2316269}) A locally compact group $G$ has property F$_\cE$ if every action of $G$ by affine isometries on a space in $\cE$ has a fixed point.
\end{definition}

It still holds that for $\sigma$-compact groups, F$_\cE$ implies (T$_\cE$), but the converse is not true. For example Pansu's computation of $L_p$-cohomology of rank $1$ symmetric spaces \cite{MR1086210} says that $Sp(n,1)$ does not have F$_{L_p}$ for $p>4n+2$, whereas as every group with property (T) \cite{MR2316269}, it has T$_{L_p}$ for every $1\leq p<\infty$. Pansu's result has been generalized by Yu \cite{MR2221161} who showed that every Gromov-hyperbolic group has a proper isometric action on an $L_p$ space for every large $p$. We refer to \cite{MR4275861,MR3590529,AminedlS,czuronKalantar} and \cite{MR3872847,MR4229199,OppenheimGarland} for recent progresses on fixed points properties for actions on $L_p$ spaces and other Banach spaces. So studying for which spaces a group has F$_X$ is a way to quantify the strength of property (T). The following conjecture therefore is another indication that $\SL_3(\Z)$ has a very strong form of property (T).
\begin{conjecture}\label{conj:BFGM}\cite{MR2316269}
  Any action by isometries of $SL_3(\Z)$ (or more generally a
  lattice in a connected simple Lie group of real rank $\geq 2$)
  on a uniformly convex Banach space has a fixed point.
\end{conjecture}

\subsection{Expander graphs}
The study of group actions on Banach spaces is also related to questions about the possible interactions between geometry of finite graphs and of Banach spaces. Given a finite graph $\cG$ and a Banach space $X$, the $X$-valued Poincar\'e constant $\rho(\cG,X)$ is the smallest constant $\rho$ such that, for every function $f$ from the vertex set of $\cG$ to $X$, the Poincar\'e inequality
\[ \inf_{\xi \in X} \|f-\xi\|_2 \leq \rho \|\nabla f\|_2\]
holds, where $\nabla f$ is the function on the egde set of $\cG$ taking the value $f(x) - f(y)$ at the edge $(y,x)$. A sequence of bounded degree graphs with size going to infinity is said to be expanding with respect to $X$ if $\inf_n \rho(\cG_n,X)>0$. When $X$ is the line, or a Hilbert space, or an $L_p$ space for $1 \leq p <\infty$, we recover the usual notion of expander graph. On the opposite, there are no expanders with respect to $\ell_\infty$, and more generally with respect to a space containing arbitrarily large copies of $\ell_\infty^n$  (such space are called \emph{spaces with trivial cotype}). A sequence $\cG_n$ that is expanding with respect to every uniformly convex Banach space is called a sequence of \emph{super-expanders} \cite{MR3210176}. There are two sources of examples known : one, that we will discuss in the last section, coming from quotients of arithmetic groups over non-archimedean local fields \cite{MR2574023} and relying on the ideas from Section~\ref{sec:SL3}, and one coming from Zig-Zag products \cite{MR3210176}.

When a sequence of graphs $\cG_n$ are Cayley graphs of $(\Gamma_n,S_n)$ where $\Gamma_n$ is a sequence finite quotients of a group $\Gamma$ with size going to infinity and $S_n$ the image of a fixed generating set of $\Gamma$ in $\Gamma_n$, the fact that $\cG_n$ are expanders with respect to $X$ is equivalent to the fact that the representation $\pi$ on $\ell_2(\cup_n \Gamma_n;X)$ has spectral gap. This if for example the case if $\Gamma$ has (T$_{\ell_2(\N;X)}$). It follows from this discussion is that Conjecture~\ref{conj:BFGM} is stronger than the following conjecture:
\begin{conjecture}\label{conj:SL3superexpanders}
  The sequence of Cayley graphs of $\SL_3(\Z/n\Z)$ with respect to the elementary matrices is a sequence of super-expanders.
\end{conjecture}
It is conceivable that they are even expanders with respect to every Banach space of nontrivial cotype. The existence of such expanders is still unknown. On the opposite, it is also unknown whether there exist expanders that are not expanders with respect to all Banach spaces of nontrivial cotype. This is revealing of how little such questions are understood.

\subsection{Strong Banach property (T)}
Let $\cE$ be a class of Banach spaces. If, in Definition~\ref{def:strongT}, we allow representations on a Banach space in $\cE$ rather than only on a Hilbert space, we say that $G$ as strong property (T) with respect to $\cE$. So there are as many Banach space versions of strong property (T) that classes of Banach spaces. In \cite{MR2574023}, Lafforgue uses the terminology strong Banach property (T) to mean that $G$ has strong property (T) with respect to every class $\cE$ in which $\ell_1$ is not finitely representable. This is essentially the largest possible class, because no non compact group can have strong property (T) with respect to $L_1(G)$.

Strong property (T) with respect to Banach spaces has the same kind of applications than for Hilbert spaces. First as in Proposition \ref{prop:H1}, strong property (T) with respect to $X\oplus \CC$ implies that $H^1(G,\pi)=0$ for every representation on $X$ with small enough exponential growth. We also have a variant of Theorem~\ref{thm:bfh}. If $G$ is assumed to have strong property (T) with respect to all subspaces of $L_p$ spaces for all $2\leq p<\infty$, then it is enough to assume that the action is by $C^{1+s}$-diffeomorphisms to ensure that for every $\varepsilon>0$ $\alpha$ preserves of metric of regularity $C^{s-\varepsilon}$ \cite{brownFisherHurtado}, and enough to assume that the action is by $C^1$-diffeomorphisms to ensure that for every $p<\infty$ $\alpha$ preserves a measurable metric that is $L^p$-integrable \cite{brownDamjanovicZhang}.

The sketch of proof of strong property (T) for $\SL_3(\R)$ and $\SL_3(\Z)$ is Section~\ref{sec:strongT} applies in the same way, provided that there are constants $C,\theta>0$ such that
\begin{equation}\label{eq:norm_TdeltaonX} \forall \delta \in [-1,1], \|T_\delta - T_0\|_{B(L_2(S^2;X))} \leq C |\delta|^{\theta/2}.
\end{equation}
\begin{theorem}\cite{MR2423763,dlSActa} $\SL_3(\R)$ and $\SL_3(\Z)$ have strong property (T) with respect to $X$ for every Banach space satisfying \eqref{eq:norm_TdeltaonX} for some $C,\theta>0$. 
\end{theorem}
It is expected that all uniformly convex Banach spaces spaces satisfy \eqref{eq:norm_TdeltaonX} for some $C,\theta$. More precisely, it should be true that spaces satisfying \eqref{eq:norm_TdeltaonX} are exactly the spaces of nontrivial Rademacher type. This would settle Conjecture \ref{conj:BFGM} and \ref{conj:SL3superexpanders}.

If follows from the Riesz-Thorin theorem that \eqref{eq:norm_TdeltaonX} holds for $L_p$ spaces for $1<p<\infty$ and therefore for every subspace of an $L_p$ space, and more generally for every $\theta$-Hilbertian space. In \cite{dlSIsrael15}, exploiting the stronger summability property from \eqref{eq:Spestimate_Tdeltaupper}, I showed that \eqref{eq:norm_TdeltaonX} holds whenever $d_n(X) = O(n^{\frac 1 4 - \varepsilon})$ as $n \to \infty$. Here $d_n(X)$ denotes the supremum over all subspaces $E$ of $X$ of dimension $n$ of the Banach-Mazur distance to the Euclidean space of the same dimension. For example, this holds if $X$ has type $p$ and cotype $q$ with $\frac 1 p - \frac 1 q<\frac 1 4$.

\section{Other groups}\label{sec:otherGroups}
There is no general theory yet in which the strategy from Section~\ref{sec:SL3} for $\SL_3(\R)$ fits, but there are other examples of groups for which such tools have been developped~: $\SL_3(\F)$ for a non-archimedean local field \cite{MR2423763}, $\SP_2(\R)$ and its universal cover \cite{MR3047470,MR3035056,MR3453357,strongTsp4}, $\SP_2(\F)$ \cite{liao,liaoAP}, $\SL_n(\F)$ and $\SL_n(\R)$ for $n \geq 4$ \cite{LaffdlS} and \cite{deLaatdlSCrelle,dlMdlSAIF}, $\SO(n,1)$ \cite{ParcetRicarddlS} and finally lattices in locally finite affine buildings of type $\tilde{A}_2$ \cite{LecureuxWitzeldlS}. 

Let me expand a bit, starting with the real Lie groups. The group $\SP_2(\R)$ is the group of $4\times 4$ matrices which preserve the standard symplectic form $\omega(x,y) = y_1 x_3 + y_2 x_4 -  x_1 y_3 - x_2 y_4$. The Bernstein inequalities used in Proposition~\ref{prop:Uinvariant_K_coeff} were generalized in \cite{MR3165537} to Jacobi polynomials, which appear as spherical functions for other Gelfand pairs than $\SO(2) \subset \SO(3)$. In \cite{MR3047470,MR3453357}, Haagerup and de Laat then generalized Theorem~\ref{thm:nonAPSL3} to $\SP_2(\R)$ and its universal cover respectively. The analogue of Theorem~\ref{thm:nonAPLpSL3} was obtained in \cite{MR3035056}, but for $p>12$ (improved to $p>10$ in \cite{strongTsp4} by refining the Bernstein inequalities from \cite{MR3165537}), and strong property (T) was proved in \cite{strongTsp4} for the Lie groups or their cocompact lattices, and in \cite{dlSActa} for their non cocompact lattices, for a class of Banach spaces that is more restritive than for $\SL_3(\R)$. By rank $2$ reduction all these results extend to all real connected semisimple Lie groups all of whose simple factors have rank $\geq 2$, and all their lattices.

When $\F$ is a non-archimedean local field, in the same way, to obtain
all almost $\F$-simple algebraic group of split rank at least two, it
was enough to consider $\SL_3$ and $\SP_2$. Lafforgue's original
article \cite{MR2423763} already contained a proof of strong property
(T) for $\SL_3(\F)$. Steps 2 and 3 are almost identical to the real
case, but the first step is very different, because maximal compact
subgroups of $\SL_3(\F)$ are very different from those in
$\SL_3(\R)$. For example when $\F=\Q_p$, a maximal compact subgroup is
$\SL_3(\Z_p)$, which contains large nilpotent groups (the groups of
upper-triangular matrices). This difference turns out to play a rôle
for Banach space versions of strong property (T). Indeed, exploiting
these large nilpotent groups and the good understanding of abelian
Fourier analysis on vector-valued $L_p$ spaces \cite{MR675400},
Lafforgue \cite{MR2574023} was able to make the first step work for
representations of the maximal compact subgroups of $\SL_3(\F)$ on
arbitrary Banach space of nontrivial type. For $\SP_2$, the same was
proved by Liao in \cite{liao}. As a consequence, all almost
$\F$-simple algebraic group of split rank at least two and their
lattices have strong property (T) with respect to Banach spaces of
nontrivial type, and Conjectures~\ref{conj:BFGM} and
\ref{conj:SL3superexpanders} hold for $\SL_3(\Z)$ replaced by such
lattices.

The Fourier analysis and approximation results from
Section~\ref{sec:fourier} have also been obtained for non-archimedean
local fields, in \cite{LaffdlS} for $\SL_3$ and \cite{liaoAP} for
$\SP_2$. The results are identical for $\SL_3(\F)$ and
$\SL_3(\R)$. Interestingly, for $\SP_2$ a difference appears : the
condition $p>10$ becomes, better, $p>4$ in the non-archimedean case
\cite{liaoAP}.

The fact that, both in the real and non-archimedean case, the proofs do not allow to take $p$ down to $2$ in Theorem~\ref{thm:nonAPLpSL3} has been a motivation for finding other groups for which the restriction on $p$ is smaller. And indeed, this is the case for $\SL_n(\F)$ \cite{LaffdlS} and $\SL_n(\R)$ \cite{deLaatdlSCrelle} for large $n$. Let us focus on $\SL_n(\R)$, because the situation is closer to Section~\ref{sec:SL3}. In that case, a satisfactory replacement for Step 1 is to work with the pair $\SO(d) \subset \SO(d+1)$. In the sphere picture we are studying the operators $T_\delta$ defined as for $S^2$ but in higher dimension $S^d = \SO(d+1)/\SO(d)$. In \cite{deLaatdlSCrelle}, we proved that the map $\delta \in (-1,1) \mapsto T_\delta \in S_p$ is H\"older continous for every $p > 2 + \frac{2}{d-1}$. Step 2 is more delicate. Taking $n \geq d+1$ and seeing $\SO(d+1) \subset \SL_n(\R)$, by considering all possible matrices $D,D' \in \SL_n(\R)$, the maps $k\in \SO(d+1)\mapsto DkD' \in \SL_n(\R)$ pass to maps from the segment $[-1,1] = \SO(d)\backslash \SO(d+1) /\SO(d)$ into the Weyl chamber $\Lambda_n:=\SO(n)\backslash \SL_n(\R) /\SO(n)$. The problem is that all these maps take values in a fixed $d-2$-codimensional subset of $\Lambda_n$, so it is hopeless to efficiently connect any two points in the Weyl chamber by such moves as in Figure~\ref{figure:zigzag}. Even worse, when $n<2d-1$, the unions of these segments has only bounded connected components. Fortunately, when $n \geq 2d-1$ these connected components merge to form an unbounded component, and a weaker form of efficient exploration of this unbounded component is possible.  Putting everything together, we obtain that $\SL_{2d-1}(\Z)$ satisfies Theorem~\ref{thm:nonAPLpSL3} for every $p > 2 + \frac{2}{d-1}$. Equivalently, the non-commutative $L_p$ space of the von Neumann algebra of $\SL_{2d-1}(\Z)$ does not have the CBAP for every $p > 2 + \frac{2}{d-1}$. If $d \geq 5$, by rank $2d-1$ reduction, we obtain that the same is true for every $\Gamma$ is a simple Lie group of rank at least $2d-1$. In particular, if $L_p(\LL \SL_3(\Z))$ had the CBAP for some $p>2$, this would distinguish the von Neumann algebras of $\SL_3(\Z)$ and $\SL_{2d-1}(\Z)$ for large $d$. Even more optimistically, if conversely $L_p(\LL \SL_{2d-1}(\Z))$ had the CBAP for $2 \leq p\leq \frac{2}{d-1}$, this would distinguish the von Neumann algebras of $\SL_n(\Z)$ for odd $n$.

In \cite{dlMdlSAIF} the exploration procedure of the Weyl chamber of $\SL_n(\R)$ was refined, and this allowed to prove strong property (T) for $\SL_n(\R)$ with respect to classes of Banach spaces that become larger with $n$~: if a Banach space $X$ satisfies that, for some $\beta<\frac 1 2$, $d_k(X) = O( k^{\beta})$ as  $k\to \infty$, then $X$ has strong property (T) with respect to $X$ for every $n \geq \frac{c}{\frac 1 2 - \beta}$. The property that $d_k(X) =o(k^{\frac 1 2})$ characterizes the Banach spaces with nontrivial type \cite{MR467255}, and it is an old problem whether they automaticall satisfy the stronger condition $d_k(X) =o(k^{\frac 1 2})$. This is for example the case if $X$ has type $2$.

So far, we have only talked about higher rank groups, and rank $0$ reduction was used to prove strong rigidity results. But rank $0$ reduction can also say something about rank $1$ groups, which are not rigid in the same way as higher rank group. The following is an example.
\begin{prop}\cite{ParcetRicarddlS} Every $\SO(n)$-biinvariant matrix coefficient of every representation of $\SO(n,1)$ on a Hilbert space is of class $C^{\frac{n}{2}-1-\varepsilon}$ outside of $\SO(n)$, for every $\varepsilon>0$.
\end{prop}
We do not know if the regularity $\frac{n}{2}-1$ is optimal, but the linear order of $n$ is, already for unitary representations. For odd $n$, $\varepsilon=0$ is allowed.





\bibliographystyle{plain}
\bibliography{biblio}








\end{document}